\documentclass[12pt]{amsart}
\usepackage{amsmath,amssymb}
\theoremstyle{plain}
\newtheorem{thm}{Theorem}[section]
\newtheorem{lem}[thm]{Lemma}
\theoremstyle{definition}

\newtheorem{ex}[thm]{Example}
\newtheorem{rem}[thm]{Remark}

\numberwithin{equation}{section}

\newcommand{\R}{\mathbb{R}}

\begin{document}
\title[Nonlocal impulsive BVPs]{Positive solutions of some nonlocal impulsive boundary value problem}
\date{}

% ----------------------------------------------------------------

\subjclass[2000]{Primary 34B37, secondary  34A37, 34B10, 34B18}%
\keywords{Fixed point index, cone, positive solution, impulsive differential equation.}%

\author{Gennaro Infante}
\address{Gennaro Infante, Dipartimento di Matematica ed Informatica, Universit\`{a} della
Calabria, 87036 Arcavacata di Rende, Cosenza, Italy}%
\email{gennaro.infante@unical.it}%

\author{Paolamaria Pietramala}%
\address{Paolamaria Pietramala, Dipartimento di Matematica ed Informatica, Universit\`{a} della
Calabria, 87036 Arcavacata di Rende, Cosenza, Italy}%
\email{pietramala@unical.it}%

\begin{abstract}
We prove new results on the existence of positive solutions for some impulsive differential equation subject to
nonlocal boundary conditions. Our boundary conditions involve an affine functional given by a Stieltjes integral.
These cover the well known multi-point boundary conditions that are studied by various authors.
\end{abstract}

\maketitle

\section{Introduction}
Differential equations with impulses arise quite often in the study of different problems,
in particular are used as a model for evolutionary processes subject to a sudden rapid
change of their state at certain moments.
The theory of impulsive differential equations has become recently a quite active area of research.
For an introduction to this theory we refer to the books
\cite{bainov, be-he-nt-book, lak-bai-sim,sam-per}, which also contain a variety of interesting examples and applications.

More recently, boundary value problems (BVPs) for impulsive second-order differential equations
 have been studied by several authors,
see for example  \cite{afo,ao,erbe-liu,erbe-kra,guo2,guo1,lee-lee1,lee-liu,lin-jiang,liu-guo,liu-liu,rach-tom,zu-ji-oreg}
and the references therein.
In particular, the existence of positive solutions under the so-called $m$-point boundary conditions
 has been investigated,
in the context of impulsive differential equations, in \cite{feng-pang, feng-xie,jan}. 
Various techniques 
are utilized in the above papers: the Leggett-Williams theorem, the Schauder fixed point theorem, the method of upper and lower solutions, 
the fixed point index on cones and, of course,  
the well-known Guo-Krasnosel'ski\u \i{}  theorem on cone-compression and cone-expansion.

In this paper, we establish new results for the existence of positive solutions
for the second order impulsive differential equation
\begin{equation}\label{de}
u''(t)+g(t)f(t,u(t))=0,\ t \in (0,1),\ t\neq \tau,
\end{equation}
\begin{equation}\label{deim}
\Delta u|_{t=\tau}=I(u(\tau)),\ \Delta u'|_{t=\tau}=\frac{I(u(\tau))}{\tau-1},
\end{equation}
subject to the nonlocal boundary conditions (BCs)
\begin{equation}\label{debc}
u(0)={\alpha}[u],\; u(1)=0.
\end{equation}
Here $\tau \in (0,1)$, $\Delta v|_{t=\tau}$ denotes the ``jump'' of $v(t)$ in $t=\tau$, that is
$$
\Delta v|_{t=\tau}=v(\tau^+)-v(\tau^-),
$$
where $v(\tau^-)$, $v(\tau^+)$ are the left and right limits of $v(t)$ in $t=\tau$,
and
$\alpha[u]$ is a positive functional given by
$$
\alpha[u]=A_{0}+\int_0^1 u(s)\,dA(s),
$$
involving a Lebesgue-Stieltjes integral.
The impulsive differential equation \eqref{de}-\eqref{deim}, under different BCs has been studied in \cite{guo1}.
The type of BC we study here is quite general and includes as special cases
$$
\alpha[u]=\sum_{i=1}^{m} \alpha_{i}u(\xi_{i})\quad \text{and}\quad
\alpha[u]=\int_{0}^{1} {\alpha}(s)u(s)\,ds,
$$
that is, multi-point and integral BCs, that are widely studied objects.
In the case of ordinary differential equations, this has been done in several papers, see for
example
 \cite{gnt,kttmna,kl07,macast,jw-seville,zy-na06,miraJMAA04} and the references therein.

The methodology here is to write the boundary value problem \eqref{de}-\eqref{debc}
as a \emph{perturbed} integral equation and we look for fixed points of an operator $T$ in a suitable cone
of positive functions in the space $PC[0,1]$.
One advantage of this approach is that we avoid lengthly calculations to determine the Green's 
function associated to the impulsive BVP.

For simplicity, we restrict our attention to the case of one impulse. In Remark \ref{two-impulses} we suggest how
this approach can be modified to work with a \emph{finite} number of impulses.

Our main ingredient is the classical fixed point index theory and in the last Section we provide an example to illustrate our theory.

\section{Existence of positive solutions of some
integral equations}\label{sec2}
 We study the existence of positive solutions of the integral equation
\begin{equation}\label{eqphamint2}
u(t)=\gamma(t){\alpha}[u] + \int_{0}^{1} k(t,s)g(s)f(s,u(s))\,ds + \gamma(t)\chi_{(\tau,1]}\frac{I(u(\tau))}{1-\tau},
\end{equation} where $t \in [0,1]$, and $\tau \in (0,1)$ is fixed. We work in the Banach space
 \begin{align*}
 PC[0,1]:=\{u:[0,1]\rightarrow\R: \;& u \;\text{is continuous in}\; t\in [0,1]\backslash\{\tau\},\\
 &\text{there exist}\; u(\tau^-)=u(\tau)\;\text{and}\;u(\tau^+)<\infty \},
 \end{align*}
 endowed with the supremum norm $\|u\|=\sup\{|u(t)|:\;t\in [0,1]\}$.

We make use of the classical fixed point index for compact maps
(see for example \cite{amann} or \cite{guolak}) on the cone
\begin{equation}\label{eqcone}
K=\bigl\{u\in PC[0,1]: u(t)\geq 0\ \forall \, t\in [0,1]\ \text{and}\ \min_{t \in [a,b]}u(t)\geq c\|u\|\bigr\},
\end{equation}
where $[a,b]$ is some subset of $(\tau,1)$ and $c$ is a positive constant.

 From now on, we assume that $f,g,
\alpha , \gamma, I$ and the kernel $k$ have the following
properties:
\begin{enumerate}
\item [$(C_{1})$] $f:[0,1]\times[0,\infty)
\to [0,\infty)$ satisfies Carath\'{e}odory conditions, that is, for each $u$, $t \mapsto f(t,u)$ is measurable
and for almost
every $t$, $u \mapsto f(t,u)$ is continuous, and  for every $r>0$ there
exists a $L^\infty$-function $\phi_{r}:[0,1]\to [0,\infty)$ such that
$$
f(t,u)\le \phi_{r}(t) \quad\text{for almost all }t\in [0,1] \;\;\text{and
all }u\in [0,r].$$
\item [$(C_{2})$] $k:[0,1]\times [0,1]\to [0,\infty)$ is measurable, and for every $t_1\in [0,1]$ we have
                  $$
                  \lim_{t \to t_1}\int_0^1 |k(t,s)-k(t_1,s)|\phi_{r}(s)\,ds=0.
                  $$
\item [$(C_{3})$] There exist  $[a,b]\subset (\tau,1)$, a  $L^\infty$-function  $\Phi:[0,1]\to
                  [0,\infty)$ and a constant $c_{1} \in (0,1]$ such that
\begin{align*}
           k(t,s)\leq \Phi(s) \text{ for } &t \in [0,1] \text{ and almost every }s\in [0,1]\\
           k(t,s) \geq c_{1}\Phi(s) \;\text{ for } &t\in [a,b] \text{ and almost every } s \in [0,1].
\end{align*}
\item [$(C_{4})$] $\gamma :[0,1]\to [0,\infty)$ is continuous and there exists a constant $c_{2}\in (0,1]$ such that
                  $$
                  \gamma (t) \geq c_{2}\|\gamma\| \;\text{ for } t\in [a,b].
                  $${}
\item [$(C_{5})$] ${\alpha}: K \to [0,\infty)$ is a continuous functional with
                  $$
                  {\alpha}[u]= A_{0}+\int_0^1 u(s)\,dA(s),
                  $$
                  where $dA$ is a positive  Lebesgue-Stieltjes measure, $A$ is continuous in $\tau$ and
                   $\int_0^1 dA(s)<\infty$.
\item [$(C_{6})$] $g\in L^1[0,1]$, $g\geq 0$ a.\,e. and $\int_a^b \Phi(s)g(s) \,ds>0$.
\item [$(C_{7})$]$I:[0,\infty)\rightarrow[0,\infty)$ is a continuous function and there exist $\delta_1,\delta_2\geq 0$ such that
 $$
 \delta_1 x\leq I(x) \leq \delta_{2}x \text{ for } x\in [0,\infty).
 $$
\end{enumerate}

Under these hypotheses we can work in the cone \eqref{eqcone},
 where $c=\min\{c_{1},c_{2}\}$ and $[a,b]$ as in $(C_{3})$.
 
If $\Omega$ is a bounded open subset of $K$ (in the relative
topology) we denote by $\overline{\Omega}$ and $\partial \Omega$
the closure and the boundary relative to $K$. We write
$$
K_{r}=\{u\in K: \|u\|<r\}  \text{ and } \overline{K}_{r}=\{u\in K:\|u\|\leq
r\}.
$$

 We consider now the  map $T: K\rightarrow PC[0,1]$ defined, for $u\in K$, by
$$
Tu(t):= \gamma(t){\alpha}[u] + \int_{0}^{1} k(t,s)g(s)f(s,u(s))\,ds
+ \gamma(t) {\chi}_{(\tau,1]}\frac{I(u(\tau))}{1-\tau}.
$$

In order to prove that $T$ is compact, we make use of the following compactness criterion, which can be found in
\cite{afo,lak-bai-sim} and is an extension of the classical Ascoli-Arzel\`{a} Theorem.
A key ingredient here is that the interval is compact.
For a compactness criterion on unbounded intervals and its applications to impulsive differential equations see 
\cite{edglgm,marino1,marino2}.

We recall that a set $S\subset PC[0,1]$ is said to be
\textit{quasi-equicontinuous} if for every $u\in S$ and for every
$\varepsilon>0$ there exists $\delta>0$ such that $t_1,t_2 \in
[0,\tau]$ (or $t_1,t_2 \in (\tau,1]$) and $|t_1-t_2|<\delta$
implies $|u(t_1)-u(t_2)|<\varepsilon.$
\begin{lem} \label{dan}
A set $S\subseteq PC[0,1]$ is relatively compact in $PC[0,1]$ if
and only if $S$ is bounded and quasi-equicontinuous.
\end{lem}
\begin{thm}\label{thmk}
 If the hypotheses $(C_{1})$-$(C_{7})$ hold for some $r>0$, then $T$ maps
 $\overline{K}_{r}$ into $K$. When these hypotheses
 hold for each $r>0$, $T$  maps $K$ into $K$. Moreover,
 $T$ is a compact map.
\end{thm}
\begin{proof}{}
 Let $u\in
\overline{K}_{r}$. Then we have, for $t \in [0,1]$,
$$
Tu(t)= \gamma(t)\Bigl({\alpha}[u]+\chi_{(\tau,1]}\frac{I(u(\tau))}{1-\tau}\Bigr)+\int_{0}^{1} k(t,s)g(s)f(s,u(s))\,ds\geq 0,
$$
furthermore 
$$
Tu(t) \leq \gamma (t)\Bigl({\alpha}[u] + \frac{I(u(\tau))}{1-\tau}\Bigr)+\int_0^1\Phi(s)g(s)f(s,u(s))\,ds
$$
and therefore we obtain
\begin{equation}\label{stellina}
\|Tu\|\leq \|\gamma\|\Bigl({\alpha}[u] + \frac{I(u(\tau))}{1-\tau}\Bigr)+\int_0^1 \Phi(s)g(s)f(s,u(s))\,ds.
\end{equation}
Then we have
\begin{multline*}
 \min_{t\in [a,b]}Tu(t) \geq c_{2}\|\gamma\|\Bigl(\alpha[u] + \frac{I(u(\tau))}{1-\tau}\Bigr)+c_{1} \int_0^1 \Phi(s)g(s)f(s,u(s))\,ds
 \\ \geq c \Bigl[\|\gamma\|\Bigl({\alpha}[u] + \frac{I(u(\tau))}{1-\tau}\Bigr)+\int_0^1 \Phi(s)g(s)f(s,u(s))\,ds \Bigr]\geq c \|Tu\|.
\end{multline*}
Hence $Tu\in K$ for every $u\in
\overline{K}_{r}$.
Now, we show that the map $T$ is compact.
Firstly, we show that $T$ sends bounded sets into bounded sets.
It is enough to see that $T(\overline{K}_{r})$ is bounded. Let $u \in \overline{K}_{r}$.
Then, for all $t \in [0,1]$, from \eqref{stellina} we have
\begin{align*}
 \|Tu\| &\leq \|\gamma\|\Bigl({\alpha}[u] + \frac{\delta_2 r}{1-\tau}\Bigr)+ \int_{0}^{1} \Phi(s)g(s)\phi_r(s)\,ds\\
 &\leq \|\gamma\|\Bigl(A_0+r\int_0^1 dA(s)  +\frac{\delta_2 r}{1-\tau}\Bigr)+M_r,
\end{align*}
for some $0\leq M_r<\infty$.

We prove now that $T$ sends bounded
sets into quasi-equicontinuous sets.\newline
It is sufficient to prove this for $t_1,t_2\in (\tau,1]$,
$t_1<t_2$ and $u \in \overline{K}_{r}$.  We have
\begin{align*}
|Tu(t_1)-Tu(t_2)|\leq & |\gamma(t_1)-\gamma(t_2)|\Bigl({\alpha}[u] +\frac{I(u(\tau))}{1-\tau}\Bigr)\\
&+\int_{0}^{1}| k(t_1,s)-k(t_2,s)|g(s)\phi_r(s)\,ds.
\end{align*}
Then $|Tu(t_1)-Tu(t_2)|\rightarrow 0$ when $t_1\rightarrow t_2$.
From Lemma~\ref{dan} we can conclude that $T$ is a compact map.
\end{proof}
Let $dB_1$ be the Dirac measure of weight ${\delta_1}/{(1-\tau)}$ in $\tau$ and let
$dB_2$ be the Dirac measure of weight ${\delta_2}/{(1-\tau)}$ in $\tau$.
We make use of the two functionals
$$
\alpha_1[u]:=A_{0}+\int_0^1
u(s)\,dA(s)+\int_{0}^{1}u(s)dB_1(s):=A_{0}+\int_0^1
u(s)\,dA_1(s),
$$

$$
\alpha_2[u]:=A_{0}+\int_0^1
u(s)\,dA(s)+\int_{0}^{1}u(s)dB_2(s):=A_{0}+\int_0^1
u(s)\,dA_2(s)
$$
and of the following numbers
$$
f^{0,\rho}:=\sup_{0 \leq u \leq \rho, \;0 \leq t \leq 1} \frac{f(t,u)}{\rho},\quad
f_{\rho, \rho/c}:=\inf_{\rho \leq u \leq \rho/c, \; a \leq t \leq b} \frac{f(t,u)}{\rho},
$$
\begin{equation}\label{eqmbigm}
\dfrac{1}{m}:= \sup_{t \in [0,1]}\int_0^1 k(t,s)g(s)\,ds, \quad
\dfrac{1}{M(a,b)}:=\inf_{t \in [a,b]}\int_a^b k(t,s)g(s)\,ds.
\end{equation}

We assume that
\begin{enumerate}
\item [$(C_{8})$] The function $t \mapsto k(t,s)$ is integrable with respect to the measure $dA_2$, that is
                  $$
                  \mathcal{K}(s):=\int_0^1 k(t,s) \,dA_2(t)
                  $$
                  is well defined.
\end{enumerate}

Firstly, we  prove that the index is $1$ on the set $K_{\rho}$.

\begin{lem} \label{ind1b}
Suppose $\Gamma:=\int_0^1 \gamma(t) \,dA_2(t) < 1$ and assume that
there exists $\rho>0$ such that $u \neq Tu$ for all $u\in
\partial K_\rho$ and
\begin{enumerate}
\item[$(\mathrm{I}^{1}_{\rho})$]\label{EqB}
the following inequality holds:
\begin{equation}\label{EqC}
  \frac{A_{0}\| \gamma \|}{(1-\Gamma)\rho}
  +f^{0,\rho}\Bigl(\frac{\| \gamma \|}{(1-\Gamma)}\int_0^1\mathcal{K}(s)g(s)\,ds
  +\dfrac{1}{m}\Bigr)\leq 1.
\end{equation}
\end{enumerate}
Then the fixed point index, $i_{K}(T,K_{\rho})$, is $1$.
\end{lem}

\begin{proof}
We show that $\lambda u\neq Tu$ for every $u \in \partial
K_{\rho}$ and for every $\lambda >1$. In fact, if there exists
$\lambda>1$ and $u \in
\partial K_{\rho}$ such that $\lambda u=Tu$ then
\begin{align*}
  \lambda u(t)&= \gamma(t)\Bigl({\alpha}[u]+\chi_{(\tau,1]}\frac{I(u(\tau))}{1-\tau}\Bigr)+\int_{0}^{1} k(t,s)g(s)f(s,u(s))\,ds
  \\
  &\leq \gamma(t)\Bigl({\alpha}[u]+\frac{\delta_2(u(\tau))}{1-\tau}\Bigr)+\int_{0}^{1} k(t,s)g(s)f(s,u(s))\,ds.
\end{align*}
Then we have
\begin{equation}\label{EqD}
 \lambda u(t)\leq \gamma(t){\alpha}_2[u]+\int_{0}^{1} k(t,s)g(s)f(s,u(s))\,ds
\end{equation}
and
\begin{equation*}\label{EqE}
  \lambda \int_0^1 u(t)dA_2(t)\leq {\alpha}_2[u] \Gamma + \int_0^1
\mathcal{K}(s)g(s)f(s,u(s))\,ds.
\end{equation*}
Hence we obtain
$$
(\lambda - \Gamma)\alpha_2 [u]\leq \lambda A_{0} + \int_0^1
\mathcal{K}(s)g(s)f(s,u(s))\,ds.
$$
Substituting into \eqref{EqD} gives
$$
\lambda u(t)\leq \frac{\lambda A_{0} \gamma (t)}{\lambda - \Gamma} +
\frac{\gamma (t)}{\lambda - \Gamma} \int_0^1
\mathcal{K}(s)g(s)f(s,u(s))\,ds + \int_0^1
k(t,s)g(s)f(s,u(s))\,ds.
$$
Taking the supremum for $t\in [0,1]$ gives
\begin{align*}
\lambda \rho &\leq \frac{\lambda A_{0}\| \gamma \|}{\lambda -
\Gamma} + \frac{\| \gamma \|}{\lambda - \Gamma} \int_0^1
\mathcal{K}(s)g(s)\rho f^{0,\rho}\,ds + \sup_{t \in [0,1]} \int_0^1
k(t,s)g(s)\rho f^{0,\rho}\,ds \\& < \frac{ A_{0}\| \gamma \|}{1- \Gamma}
+ \frac{\| \gamma \|}{1- \Gamma} \int_0^1
\mathcal{K}(s)g(s)\rho f^{0,\rho}\,ds + \sup_{t \in [0,1]} \int_0^1
k(t,s)g(s)\rho f^{0,\rho}\,ds.
 \end{align*}
  Thus we have,
\begin{equation*}\label{EqF}
\lambda < \frac{A_{0}\| \gamma \|}{(1-\Gamma)\rho}
  +f^{0,\rho}\Bigl(\frac{\| \gamma \|}{(1-\Gamma)}\int_0^1\mathcal{K}(s)g(s)\,ds
  +\dfrac{1}{m}\Bigr)\leq 1.
\end{equation*}
This contradicts the fact that $\lambda>1$ and proves the result.
\end{proof}
We make use of the open set
$$
V_\rho=\bigl\{u \in K: \min_{ t\in
[a,b]}u(t)<\rho \bigr\}.
$$
 $V_\rho$ is similar to the set
called $\Omega_{\rho /c}$ in \cite{kljlms}. Note that $K_{\rho}\subset V_{\rho}\subset K_{\rho/c}$.

We now prove that the index is 0 on the set $V_{\rho}$.
\begin{lem} \label{idx0b}
Assume that there exists $\rho>0$ such that $u \neq Tu$ for $u\in
\partial V_\rho$ and
\begin{enumerate}
\item[$(\mathrm{I}^{0}_{\rho})$] the following inequalities hold:
\begin{equation}\label{alphid0}
{\alpha}_1[u]\geq {\alpha}_{0} \rho \text{ for } u\in \partial
V_\rho,
\end{equation}
where $\alpha_0\geq 0$, and
\begin{equation}\label{eqind0}
c_{2}\|\gamma\| \alpha_{0} + \dfrac{1}{M(a,b)}
f_{\rho, \rho/c}\geq 1.
\end{equation}
\end{enumerate}
Then we have $i_{K}(T,V_{\rho})=0$.
\end{lem}

\begin{proof}
Let $e(t)\equiv 1$ for $t\in [0,1]$. Then $e\in K$. We prove that
\begin{equation*}
u\neq T(u)+\lambda e\quad\text{for all } u\in \partial
V_{\rho}\quad\text{and } \lambda>0.
\end{equation*}
 In fact, if this does not happen, there exist $u\in \partial V_\rho$ and
$\lambda>0$ such that $u=Tu+\lambda e$. We have, for all $t\in [a,b]$
\begin{align*}
u(t)&= \gamma(t){\alpha}[u] + \int_{0}^{1} k(t,s)g(s)f(s,u(s))\,ds + \gamma(t)\chi_{(\tau,1]}\frac{I(u(\tau))}{1-\tau}+\lambda\\
&=\gamma(t)\Bigl({\alpha}[u] + \frac{I(u(\tau))}{1-\tau}\Bigr)+
\int_{0}^{1}k(t,s)g(s)f(s,u(s))\,ds + \lambda \\
&\geq \gamma(t)\Bigl({\alpha}[u] + \frac{\delta_1 u(\tau)}{1-\tau}\Bigr) +
\int_{0}^{1}k(t,s)g(s)f(s,u(s))\,ds + \lambda \\
&= \gamma(t){\alpha}_1[u] +
\int_{0}^{1}k(t,s)g(s)f(s,u(s))\,ds + \lambda \\
 &\ge
c_{2}\|\gamma\|{\alpha}_{0}\rho+\rho\int_a^b
k(t,s)g(s)f_{\rho, \rho/c}\,ds+\lambda\\
&\geq \rho\Bigl(c_{2}\|\gamma\| \alpha_{0} + \dfrac{1}{M(a,b)}
f_{\rho, \rho/c}\Bigr)+\lambda.
\end{align*}
By $(\mathrm{I}^{0}_{\rho})$, this implies that
$$
 \min_{t\in [a,b]}u(t)\geq \rho\Bigl(c_{2}\|\gamma\| \alpha_{0} + \dfrac{1}{M(a,b)}
f_{\rho, \rho/c}\Bigr)+\lambda \geq \rho+\lambda>\rho,
$$
contradicting the fact that $u\in \partial V_\rho$.
\end{proof}

Note that, by means of the two Lemmas above, one may also provide a result on the existence
of multiple positive solutions. In fact, if the nonlinearity $f$ has a
suitable oscillatory behavior, by nesting in an appropriate
way several
$V_{\rho}$'s and $K_{\rho}$'s, one may establish the existence of
 multiple positive solutions (we refer the reader to \cite{gijwems,kljlms}
to see the type of results that may be stated).
Here, for brevity, we state a result for the case of
one positive solution for Eq.~\eqref{eqphamint2}.

\begin{thm}\label{thmmsol1}
Eq.~\eqref{eqphamint2} has a positive solution in $K$ if either of
the following conditions hold.
\begin{enumerate}
\item[$(H_{1})$] There exist $\rho_{1},\rho_{2}\in (0,\infty)$ with
$\rho_{1}<\rho_2$ such that
$(\mathrm{I}^{1}_{\rho_{1}}), (\mathrm{I}^{0}_{\rho_{2}})$ hold.
  \item[$(H_{2})$] There exist
$\rho_{1},\rho_{2}\in (0,\infty)$ with
$ \rho_{1}<c\rho_{2}$ such that
$(\mathrm{I}^{0}_{\rho_{1}})$,
$(\mathrm{I}^{1}_{\rho_{2}})$ hold.
\end{enumerate}
\end{thm}
We omit the proof which follows simply from properties of fixed
point index, for details of similar proofs see
\cite{gijwjmaa,kljdeds}.
\begin{rem}\label{two-impulses}
So far we have discussed the case of having the impulse in just one point $\tau \in (0,1)$. 
Similar arguments work in the case of a \emph{finite} number of impulses.

For example, in the case of two impulses, say 
$$ 
\Delta u|_{t=\tau_1}=I_1(u(\tau_1)),\quad \Delta 
u|_{t=\tau_2}=I_2(u(\tau_2)), 
$$ 
$$ 
\Delta u'|_{t=\tau_1}=\frac{I_1(u(\tau_1))}{\tau_1-1},\quad \Delta 
u'|_{t=\tau_2}=\frac{I_2(u(\tau_2))}{\tau_2-1}, 
$$ 

where $0<\tau_1<\tau_2<1$, one may work in the space (with an abuse of 
notation) 
 \begin{align*} 
 PC[0,1]:=\{u:[0,1]\rightarrow\R: \;& u \;\text{is continuous in}\; t\in 
[0,1]\backslash\{\tau_1,\tau_2\},\\ 
 &\text{there exist}\; 
u(\tau_i^-)=u(\tau_i)\;\text{and}\;u(\tau_i^+)<\infty,\;i=1,2 \}, 
 \end{align*} 
 and seek for fixed points of the operator 
 $$ 
\tilde{T}u(t):= \gamma(t)\Bigl({\alpha}[u] 
+{\chi}_{(\tau_1,1]}\frac{I_1(u(\tau_1))}{1-\tau_1}+ 
{\chi}_{(\tau_2,1]}\frac{I_2(u(\tau_2))}{1-\tau_2} \Bigr)+\int_{0}^{1} 
k(t,s)g(s)f(s,u(s))\,ds. 
$$ 
in the cone \eqref{eqcone}, where 
 $[a,b]$ this time is a subset of $(\tau_2,1)$. 

Thus, if there exist positive constants 
$\delta_{1,1},\delta_{1,2},\delta_{2,1},\delta_{2,2}$ such that 
 $$ 
 \delta_{1,i} x\leq I_i(x) \leq \delta_{2,i}x \text{ for } x\in [0,\infty) 
\; \text{and}\; i=1,2, 
 $$ 
 one may consider the measures $d\tilde{B}_1$ and $d\tilde{B}_2$, where 
$d\tilde{B}_1$ is the Dirac measure of weight ${\delta_{1,1}}/{(1-\tau_1)}$ 
in $\tau_1$ and of weight 
 ${\delta_{1,2}}/{(1-\tau_2)}$, 
and 
$d\tilde{B}_2$ is the Dirac measure of weight ${\delta_{2,1}}/{(1-\tau_1)}$ 
in $\tau_1$ and of weight 
${\delta_{2,2}}/{(1-\tau_2)}$ in $\tau_2$. 

The above can be used to provide a modified version of Lemmas 
\eqref{ind1b}-\eqref{idx0b} in this new context.
\end{rem}

\section{Positive solutions of the impulsive BVP.}
We now consider the BVP
\begin{equation}\label{de2}
u''(t)+g(t)f(t,u(t))=0,\ t \in (0,1),\ t\neq \tau,
\end{equation}
\begin{equation}\label{deim2}
\Delta u|_{t=\tau}=I(u(\tau)),\ \Delta u'|_{t=\tau}=\frac{I(u(\tau))}{\tau-1},
\end{equation}
\begin{equation}\label{debc2}
u(0)={\alpha}[u],\; u(1)=0,
\end{equation}
and we associate to this BVP the integral equation 
\begin{equation}\label{int}
u(t)=\gamma(t){\alpha}[u] + \int_{0}^{1} k(t,s)g(s)f(s,u(s))\,ds +
\gamma(t)\chi_{(\tau,1]}\frac{I(u(\tau))}{1-\tau},
\end{equation}
 where
\begin{equation*}
\gamma (t)=1-t\;\text{for all}\; t\in[0,1],\;\text{and}\;k(t,s)=\begin{cases} s(1-t),\ &s\leq t\\ t(1-s),\ &s>t.
\end{cases}
\end{equation*}

By a
solution of the BVP \eqref{de2}-\eqref{debc2} we mean a solution $u\in
PC[0,1]$ of the corresponding integral equation \eqref{int}.

Here may choose
$$
\Phi(s)=s(1-s).
$$
Therefore, we can take $[a,b]\subset (\tau,1)$ and
\begin{equation}\label{c}
c:= \min \{ a, 1-b \}.
\end{equation}
Now $(C_{2}),(C_{3}),(C_{4})$ are satisfied, \eqref{eqind0} reads more simply
\begin{equation}\label{gind0}
c_{2} \alpha_{0} +f_{{\rho},{\rho / c}} \cdot\frac{1}{M(a,b)}\geq 1,
\end{equation}
and \eqref{EqC} reads 
\begin{equation}\label{EqH}
  \frac{A_{0}}{\rho(1-\Gamma)}
  +\Bigl( \frac{1}{1-\Gamma}\int_{0}^{1}\mathcal{K}(s)g(s)\,ds
  +\frac{1}{m} \Bigl) f^{0,\rho} \leq 1.
\end{equation}
\begin{ex}
We now assume that $g \equiv 1$,
$\alpha [u]=\alpha u(\xi)$, with $\alpha >0$ and $\xi\in (\tau, 1)$. In this case we have
$$
\alpha_1[u]=\alpha u(\xi)+\dfrac{\delta_1}{1-\tau} u(\tau),\quad
\alpha_2[u]=\alpha u(\xi)+\dfrac{\delta_2}{1-\tau} u(\tau).
$$
We may take $A_0=0$
and $dA_2$ the Dirac measure of weight $\alpha$ in $\xi$ and of
weight ${\delta_2}/{(1-\tau)}$ in $\tau$. Thus
$$
\Gamma:=\int_0^1 \gamma(t) \,dA_2(t)=\alpha (1-\xi)+\delta_2
$$
and
$$
\int_0^1 \mathcal{K}(s)\,d(s)=\dfrac{\alpha}{2}\xi(1-\xi
)+\dfrac{\delta_2}{2} \tau.
$$
We can take $[a,b]$ such that $\xi \in [a,b]$.
Then we can set $\alpha_0=\alpha$, since $\alpha_1[u]\geq \alpha u(\xi)\geq \alpha \rho$ for
$u\in \partial V_\rho$. In this case \eqref{gind0} reads
\begin{equation}\label{gind02}
(1-b) \alpha +f_{{\rho},{\rho / c}} \cdot\frac{1}{M(a,b)}\geq 1,
\end{equation}
and \eqref{EqH} reads
\begin{equation}\label{EqH2}
 \Bigl( \frac{\alpha\xi(1-\xi
)+\delta_2\tau}{2-2(\alpha (1-\xi)+\delta_2)}
  +\frac{1}{m} \Bigl) f^{0,\rho} \leq 1.
\end{equation}

We now show that all the constants that appear in \eqref{gind02}
and \eqref{EqH2} can be computed.\newline
If we take $\tau={1}/{5}$, $\xi={1}/{2}$, $\alpha={4}/{5}$, $\delta_2={1}/{2}$, we can choose $[a,b]=[1/4,3/4]$.
This gives $m=8$, $M(a,b)=16$ and $c=1/4$. Therefore our requirements are
$f_{{\rho},{\rho/c}}\geq \frac{64}{5}$ and  $f^{0,\rho}\leq \frac{8}{13}$.
\end{ex}

\end{document}